\documentclass[11pt]{amsart}
\usepackage{amsfonts}
\allowdisplaybreaks[4]
\newtheorem{theorem}{Theorem}
\theoremstyle{plain}

\newtheorem{lemma}{Lemma}

\newtheorem{remark}{Remark}

\numberwithin{equation}{section}
\newtheorem{theoremalph}{Theorem}

\newtheorem{definitionalph}{Definition}

\begin{document}
\title[Some Simpson type inequalities for convex functions]{Some Simpson type inequalities for $h-$convex and $(\alpha,m)-$convex functions}
\author{Wenjun Liu}
\address{College of Mathematics and Statistics, Nanjing University of
Information Science and Technology, Nanjing 210044, China}
\email{wjliu.cn@gmail.com}

\begin{abstract}
In this paper, we establish some Simpson type inequalities for
functions whose third derivatives in the absolute value are $h-$convex and $(\alpha,m)-$convex, respectively.

\end{abstract}

\maketitle

\section{Introduction}

The following inequality is well known in the literature as Simpson's inequality:
\begin{equation}\label{e1}
\left|\int_{a}^{b}f(t)dt-\frac{b-a}{3}\left[\frac{f(a)+f(b)}{2}+2f\left(\frac{a+b}{2}\right)\right]\right|
\leq\frac{1}{2880}\|f^{(4)}\|_{\infty}(b-a)^5,
\end{equation}
where the mapping $f : [a, b] \rightarrow\mathbb{R}$ is supposed to be four time differentiable on the interval
$(a,b)$  and having the fourth derivative bounded on $(a,b)$, that is $\|f^{(4)}\|_{\infty}= {\sup_{x\in(a,b)}}|f^{(4)}(x)|
$ $ <\infty$.
This inequality gives an error bound for the classical
Simpson quadrature formula, which, actually, is one of the most used quadrature
formulae in practical applications.
In recent years, such inequalities were studied
extensively by many researchers and numerious generalizations, extensions
and variants of them appeared in a number of papers (see \cite{a20111,d1999s,d2000s,hn2010,l2005,l2011,sso2010,tyd2007}).

Let us recall definitions of some  kinds of convexity as follows.

\begin{definitionalph}
 \textit{[\ref{god}] We say that }$f:I\rightarrow
\mathbb{R}
$\textit{\ is Godunova-Levin function or that }$f$\textit{\ belongs to the
class }$Q\left( I\right) $\textit{\ if }$f$\textit{\ is non-negative and for
all }$x,y\in I$\textit{\ and }$t\in \left( 0,1\right) $\textit{\ we have \ \
\ \ \ \ \ \ \ \ \ \ \ }%
\begin{equation}
f\left( tx+\left( 1-t\right) y\right) \leq \frac{f\left( x\right) }{t}+\frac{%
f\left( y\right) }{1-t}.  \label{103}
\end{equation}
\end{definitionalph}

\begin{definitionalph}
\textit{[\ref{dr1}] We say that }$f:I\subseteq
\mathbb{R}
\rightarrow
\mathbb{R}
$\textit{\ is a }$P-$\textit{function or that }$f$\textit{\ belongs to the
class }$P\left( I\right) $\textit{\ if }$f$\textit{\ is non-negative and for
all }$x,y\in I$\textit{\ and }$t\in \left[ 0,1\right]$\textit{\ we have}%
\begin{equation}
f\left( tx+\left( 1-t\right) y\right) \leq f\left( x\right) +f\left(
y\right).  \label{104}
\end{equation}
\end{definitionalph}

\begin{definitionalph}
\textit{[\ref{hud}] Let }$s\in \left( 0,1\right] .$\textit{\ A function }$%
f:\left( 0,\infty \right] \rightarrow \left[ 0,\infty \right] $\textit{\ is
said to be }$s-$\textit{convex in the second sense if \ \ \ \ \ \ \ \ \ \ \
\ }%
\begin{equation}
f\left( tx+\left( 1-t\right) y\right) \leq t^{s}f\left( x\right) +\left(
1-t\right) ^{s}f\left( y\right),  \label{105}
\end{equation}%
\textit{for all }$x,y\in \left( 0,b\right] $\textit{\ \ and }$t\in \left[ 0,1%
\right] $\textit{. This class of }$\mathit{s-}$\textit{convex functions is
usually denoted by }$K_{s}^{2}$\textit{.}
\end{definitionalph}

\begin{definitionalph}
 \textit{[\ref{var}] Let }$h:J\subseteq
\mathbb{R}
\rightarrow
\mathbb{R}
$\textit{\ be a positive function. We say that }$f:I\subseteq
\mathbb{R}
\rightarrow
\mathbb{R}
$\textit{\ is }$h-$\textit{convex function, or that }$f$\textit{\ belongs to
the class }$SX\left( h,I\right) $\textit{, if }$f$\textit{\ is non-negative
and for all }$x,y\in I$\textit{\ and }$t\in \left[ 0,1\right] $\textit{\ we
have \ \ \ \ \ \ \ \ \ \ \ \ }%
\begin{equation}
f\left( tx+\left( 1-t\right) y\right) \leq h\left( t\right) f\left( x\right)
+h\left( 1-t\right) f\left( y\right).  \label{106}
\end{equation}
If inequality (\ref{106}) is reversed, then $f$ is said to be $h-$%
concave, i.e. $f\in SV\left( h,I\right) $.
\end{definitionalph}

 Obviously, if $h\left( t\right)
=t $, then all non-negative convex functions belong to $SX\left( h,I\right) $%
\ and all non-negative concave functions belong to $SV\left( h,I\right) $; if
$h\left( t\right) =\frac{1}{t}$, then $SX\left( h,I\right) =Q\left( I\right)
$; if $h\left( t\right) =1$, then $SX\left( h,I\right) \supseteq P\left(
I\right) $; and if $h\left( t\right) =t^{s}$, where $s\in \left( 0,1\right] $, then $SX\left( h,I\right) \supseteq K_{s}^{2}$.
For recent results  concerning $h-$convex functions see \cite{bom,bh2011,zeki,V}
and references therein.

\begin{definitionalph}  \cite{T}
\label{def 1.1} The function $f:[0,b]\rightarrow
\mathbb{R}
$ is said to be $m-$convex, where $m\in \lbrack 0,1],$ if for every $x,y\in
\lbrack 0,b]$ and $t\in \lbrack 0,1]$ we have
\begin{equation*}
f(tx+m(1-t)y)\leq tf(x)+m(1-t)f(y).
\end{equation*}%
Denote by $K_{m}(b)$ the set of the $m-$convex functions on $[0,b]$ for
which $f(0)\leq 0.$
\end{definitionalph}

\begin{definitionalph}  \cite{MIH}
The function $f: [0,b]\rightarrow\mathbb{R},$ $b>0$ is said to be $(\alpha ,m)-$convex, where $(\alpha ,m)\in \lbrack
0,1]^{2},$ if for all $x,y\in \lbrack 0,b]$ and $t\in \lbrack 0,1]$ we have%
\begin{equation*}
f(tx+m(1-t)y)\leq t^{\alpha }f(x)+m(1-t^{\alpha })f(y).
\end{equation*}%
Denote by $K_{m}^{\alpha }(b)$ the class of all $(\alpha ,m)-$convex
functions on $[0,b]$ for which $f(0)\leq 0.$
\end{definitionalph}

If we choose $(\alpha ,m)=(1,m)$, it can be easily seen that $(\alpha ,m)-$convexity reduces to $m-$%
convexity and for $(\alpha ,m)=(1,1),$ we have ordinary convex functions on $%
[0,b].$

Recently,  \"  Ozdemir et al. \cite{oak2011} established some Simpson type inequalities for
functions whose third derivatives in the absolute value are $m-$convex. In \cite{oak2012},
\" Ozdemir et al.  established the following inequalities for
functions whose third derivatives in the absolute value are $s-$convex in
the second sense.

\begin{theoremalph}
\label{tha} Let $f:I\subset \lbrack 0,\infty )\rightarrow
\mathbb{R}
$ be a differentiable function on $I^{\circ }$ such that $f'''\in L_1[a,b],$ where $a,b\in I^{\circ }$ with $a<b.$ If $%
\left| f'''\right| $ is $s-$convex in
the second sense on $[a,b]$   for some fixed $s\in (0,1],$ then
\begin{align}
&\left| \int_{a}^{b}f(x)dx-\frac{b-a}{6}\left[ f(a)+4f\left( \frac{a+b}{%
2}\right) +f(b)\right] \right| \notag \\
\leq &\frac{\left( b-a\right) ^{4}}{6}\left[ \frac{%
2^{-4-s}\left((1+s)(2+s)+34+2^{4+s}(-2+s)+11s+s^{2}\right)}{(1+s)(2+s)(3+s)(4+s)}\right]
  \left[ \left| f'''(a)\right|
+\left| f'''(b)\right| \right]. \label{1.6}
\end{align}
\end{theoremalph}

\begin{theoremalph}
\label{thb} Let $f:I\subset \lbrack 0,\infty )\rightarrow
\mathbb{R}
$ be a differentiable function on $I^{\circ }$ such that $f'''\in L_1[a,b],$ where $a,b\in I^{\circ }$ with $a<b.$ If $%
\left| f'''\right| ^{q}$ is $s-$convex
in the second sense on $[a,b]$   for some fixed $s\in (0,1]$ and $q>1$
with $\frac{1}{p}+\frac{1}{q}=1,$ then
\begin{align}
&\left| \int_{a}^{b}f(x)dx-\frac{b-a}{6}\left[ f(a)+4f\left( \frac{a+b}{%
2}\right) +f(b)\right] \right| \notag\\
\leq &\frac{\left( b-a\right) ^{4}}{48}\left( \frac{1}{2}\right) ^{\frac{1}{%
p}}\left( \frac{\Gamma (2p+1)\Gamma (p+1)}{\Gamma (3p+2)}\right) ^{\frac{1}{p%
}}  \left\{ \left[ \frac{1}{2^{s+1}(s+1)}\left| f'''(a)\right| ^{q}+\frac{2^{s+1}-1}{2^{s+1}(s+1)}\left|
f'''(b)\right| ^{q}\right] ^{\frac{1}{q}%
}\right. \notag\\
&\left. +\left[ \frac{2^{s+1}-1}{2^{s+1}(s+1)}\left| f'''(a)\right| ^{q}+\frac{1}{2^{s+1}(s+1)}\left|
f'''(b)\right| ^{q}\right] ^{\frac{1}{q}%
}\right\}.  \label{1.7}
\end{align}
\end{theoremalph}
\begin{theoremalph}
\label{thc} Suppose that all the assumptions of Theorem \ref{thb}
are satisfied. Then%
\begin{align}
&\left| \int_{a}^{b}f(x)dx-\frac{b-a}{6}\left[ f(a)+4f\left( \frac{a+b}{%
2}\right) +f(b)\right] \right| \notag\\
\leq &\frac{\left( b-a\right) ^{4}}{6}\left( \frac{1}{192}\right) ^{1-\frac{%
1}{q}} \notag\\
&\times \left\{ \left( \frac{2^{-4-s}}{(3+s)(4+s)}\left| f'''(a)\right| ^{q}+\frac{2^{-4-s}\left(
34+2^{4+s}(-2+s)+11s+s^{2}\right) }{(1+s)(2+s)(3+s)(4+s)}\left|
f'''(b)\right| ^{q}\right) ^{\frac{1}{q}%
}\right. \notag\\
&\left. +\left( \frac{2^{-4-s}\left( 34+2^{4+s}(-2+s)+11s+s^{2}\right) }{%
(1+s)(2+s)(3+s)(4+s)}\left| f'''(a)\right| ^{q}+\frac{2^{-4-s}}{(3+s)(4+s)}\left| f'''(b)\right| ^{q}\right) ^{\frac{1}{q}}\right\}.  \label{1.8}
\end{align}
\end{theoremalph}

The main purpose of this paper is to establish some new Simpson type inequalities for
functions whose third derivatives in the absolute value are $h-$convex and $(\alpha,m)-$convex, respectively.

\section{Simpson type inequalities for $h$-convex functions}

To prove our main theorems, we need the following identity  established   in  \cite{ah2011}:

\begin{lemma}
\label{L1} Let $f:I\rightarrow
\mathbb{R}
$ be a function such that $f'''$be absolutely
continuous on $I^{\circ }$, the interior of I. Assume that $a,b\in I^{\circ
},$ with $a<b$ and $f'''\in L_1[a,b].$ Then, the
following equality holds:%
\begin{align*}
 \int_{a}^{b}f(x)dx-\frac{b-a}{6}\left[ f(a)+4f\left( \frac{a+b}{2}\right)
+f(b)\right]
= \left( b-a\right) ^{4}\int_{0}^{1}p(t)f'''(ta+(1-t)b)dt,
\end{align*}%
where%
\begin{equation*}
p(t)=\left\{
\begin{array}{ll}
\frac{1}{6}t^{2}\left( t-\frac{1}{2}\right), &t\in \lbrack 0,%
\frac{1}{2}], \medskip
\\
\frac{1}{6}(t-1)^{2}\left( t-\frac{1}{2}\right), &t\in (\frac{1}{2%
}, 1].
\end{array}%
\right.
\end{equation*}
\end{lemma}

Using this lemma, we can obtain the following  inequalities for $h-$convex functions.

\begin{theorem}
\label{T1} Let $h:J\subseteq\mathbb{R} \rightarrow \mathbb{R}$  $([0,1]\subseteq J)$ be a non-negative function, and $f:I\subset \lbrack 0,\infty )\rightarrow
\mathbb{R}
$ be a differentiable function on $I^{\circ }$ such that $f''' \in L_1[a,b],$ where $a,b\in I^{\circ }$ with $a<b.$ If $%
\left| f''' \right| $ is $h-$convex  on $[a,b]$,  then
\begin{align}
&\left| \int_{a}^{b}f(x)dx-\frac{b-a}{6}\left[ f(a)+4f\left( \frac{a+b}{%
2}\right) +f(b)\right] \right| \notag \\
\leq &\frac{\left( b-a\right) ^{4}}{6}\left[ \int_{0}^{\frac{1}{2}%
}t^{2}\left( \frac{1}{2}-t\right)   h(t) dt+  \int_{0}^{\frac{1}{2}%
}t^{2}\left( \frac{1}{2}-t\right)   h(1-t) dt\right]
  \left[ \left| f'''(a)\right|
+\left| f'''(b)\right| \right]. \label{2.2}
\end{align}
\end{theorem}

\begin{proof}
From  Lemma \ref{L1} and $s-$convexity of $\left| f'''\right| $, we have%
\begin{align*}
&\left| \int_{a}^{b}f(x)dx-\frac{b-a}{6}\left[ f(a)+4f\left( \frac{a+b}{%
2}\right) +f(b)\right] \right| \\
\leq &\left( b-a\right) ^{4}\left\{ \int_{0}^{\frac{1}{2}}\left| \frac{1%
}{6}t^{2}\left( t-\frac{1}{2}\right) \right| \left| f'''(ta+(1-t)b)\right| dt\right. \\
&\left. +\int_{\frac{1}{2}}^{1}\left| \frac{1}{6}(t-1)^{2}\left( t-%
\frac{1}{2}\right) \right| \left| f'''(ta+(1-t)b)\right| dt\right\} \\
\leq &\frac{\left( b-a\right) ^{4}}{6}\left\{ \int_{0}^{\frac{1}{2}%
}t^{2}\left( \frac{1}{2}-t\right) \left( h(t)\left| f'''(a)\right| +h(1-t)\left| f'''(b)\right| \right) dt\right. \\
&\left. +\int_{\frac{1}{2}}^{1}(t-1)^{2}\left( t-\frac{1}{2}\right) \left(
h(t)\left| f'''(a)\right|
+h(1-t)\left| f'''(b)\right| \right)
dt\right\} \\
=&\frac{\left( b-a\right) ^{4}}{6}\left[\int_{0}^{\frac{1}{2}%
}t^{2}\left( \frac{1}{2}-t\right)   h(t) dt+  \int_{\frac{1}{2}}^{1}(t-1)^{2}\left( t-\frac{1}{2}\right)  h(t) dt\right]
  \left[ \left| f'''(a)\right|
+\left| f'''(b)\right| \right] ,
\end{align*}%
where we have used the fact that
\begin{align*}
&\int_{0}^{\frac{1}{2}%
}t^{2}\left( \frac{1}{2}-t\right)   h(t) dt+  \int_{\frac{1}{2}}^{1}(t-1)^{2}\left( t-\frac{1}{2}\right)  h(t) dt\\
=&\int_{0}^{\frac{1}{2}%
}t^{2}\left( \frac{1}{2}-t\right)   h(1-t) dt+  \int_{\frac{1}{2}}^{1}(t-1)^{2}\left( t-\frac{1}{2}\right)  h(1-t) dt\\
=&\int_{0}^{\frac{1}{2}%
}t^{2}\left( \frac{1}{2}-t\right)   h(t) dt+  \int_{0}^{\frac{1}{2}%
}t^{2}\left( \frac{1}{2}-t\right)   h(1-t) dt.
\end{align*}
Hence,    the  proof of   \eqref{2.2} is complete.
\end{proof}

\begin{remark}
\label{R1}
  In Theorem \ref{T1}, if we choose $h(t)=t^s$, $s\in (0,1],$  then  \eqref{2.2}
reduces to \eqref{1.6}.
\end{remark}

\begin{theorem}
\label{T2} Let $h:J\subseteq\mathbb{R} \rightarrow \mathbb{R}$  $([0,1]\subseteq J)$ be a non-negative   function, and  $f:I\subset \lbrack 0,\infty )\rightarrow
\mathbb{R}
$ be a differentiable function on $I^{\circ }$ such that $f'''\in L_1[a,b],$ where $a,b\in I^{\circ }$ with $a<b.$ If $%
\left| f'''\right| ^{q}$ is $h-$convex
 on $[a,b]$   and $q>1$
with $\frac{1}{p}+\frac{1}{q}=1,$ then
\begin{align}
&\left| \int_{a}^{b}f(x)dx-\frac{b-a}{6}\left[ f(a)+4f\left( \frac{a+b}{%
2}\right) +f(b)\right] \right| \notag\\
\leq &\frac{\left( b-a\right) ^{4}}{48}\left( \frac{1}{2}\right) ^{\frac{1}{%
p}}\left( \frac{\Gamma (2p+1)\Gamma (p+1)%
}{ \Gamma (3p+2)}\right) ^{\frac{1}{p}} \notag\\
& \times  \left\{ \left[ \left(\int_{0}^{\frac{1}{2}}  h(t)dt\right)\left|
f'''(a)\right| ^{q}+\left(\int_{0}^{\frac{1}{2}}h(1-t)dt\right)\left|
f'''(b)\right| ^{q} \right] ^{%
\frac{1}{q}}\right. \notag\\
&\left. +\left[ \left(\int_{0}^{\frac{1}{2}}h(1-t)dt\right)\left| f'''(a)\right| ^{q}+\left(\int_{0}^{\frac{1}{2}}h(t)dt\right)\left| f'''(b)\right| ^{q} \right] ^{\frac{1}{q}%
}\right\}.   \label{2.3}
\end{align}
\end{theorem}

\begin{proof}
From Lemma \ref{L1},  and using the $s-$convexity of $\left| f'''\right| ^{q}$ and the well-known H\"{o}lder's
inequality,  we have%
\begin{align*}
&\left| \int_{a}^{b}f(x)dx-\frac{b-a}{6}\left[ f(a)+4f\left( \frac{a+b}{%
2}\right) +f(b)\right] \right| \\
\leq &\frac{\left( b-a\right) ^{4}}{6}\left\{ \left( \int_{0}^{\frac{1}{2}%
}\left( t^{2}\left( \frac{1}{2}-t\right) \right) ^{p}dt\right) ^{\frac{1}{p}%
}\left( \int_{0}^{\frac{1}{2}}\left| f'''(ta+(1-t)b)\right| ^{q}dt\right) ^{\frac{1}{q}}\right. \\
&\left. +\left( \int_{\frac{1}{2}}^{1}\left( (t-1)^{2}\left( t-\frac{1}{2}%
\right) \right) ^{p}dt\right) ^{\frac{1}{p}}\left( \int_{\frac{1}{2}%
}^{1}\left| f'''(ta+(1-t)b)\right|
^{q}dt\right) ^{\frac{1}{q}}\right\} \\
\leq &\frac{\left( b-a\right) ^{4}}{6}\left( \frac{\Gamma (2p+1)\Gamma (p+1)%
}{2^{3p+1}\Gamma (3p+2)}\right) ^{\frac{1}{p}}   \left\{ \left( \int_{0}^{\frac{1}{2}}\left[ h(t)\left|
f'''(a)\right| ^{q}+h(1-t)\left|
f'''(b)\right| ^{q}\right] dt\right) ^{%
\frac{1}{q}}\right. \\
&\left. +\left( \int_{\frac{1}{2}}^{1}\left[ h(t)\left| f'''(a)\right| ^{q}+h(1-t)\left| f'''(b)\right| ^{q}\right] dt\right) ^{\frac{1}{q}%
}\right\}\\
\leq &\frac{\left( b-a\right) ^{4}}{48}\left( \frac{1}{2}\right) ^{\frac{1}{%
p}}\left( \frac{\Gamma (2p+1)\Gamma (p+1)%
}{ \Gamma (3p+2)}\right) ^{\frac{1}{p}} \\
& \times  \left\{ \left[ \left(\int_{0}^{\frac{1}{2}}  h(t)dt\right)\left|
f'''(a)\right| ^{q}+\left(\int_{0}^{\frac{1}{2}}h(1-t)dt\right)\left|
f'''(b)\right| ^{q} \right] ^{%
\frac{1}{q}}\right. \\
&\left. +\left[ \left(\int_{\frac{1}{2}}^{1}  h(t)dt\right)\left| f'''(a)\right| ^{q}+\left(\int_{\frac{1}{2}}^{1}h(1-t)dt\right)\left| f'''(b)\right| ^{q} \right] ^{\frac{1}{q}%
}\right\},
\end{align*}%
where we have used the fact that
\begin{equation}
\int_{0}^{\frac{1}{2}}\left( t^{2}\left( \frac{1}{2}-t\right) \right)
^{p}dt=\int_{\frac{1}{2}}^{1}\left( (t-1)^{2}\left( t-\frac{1}{2}\right)
\right) ^{p}dt=\frac{\Gamma (2p+1)\Gamma (p+1)}{2^{3p+1}\Gamma (3p+2)}   \label{2.3'}
\end{equation}%
and $\Gamma $ is the Gamma function.
Hence,    the  proof of   \eqref{2.3} is complete.
\end{proof}

\begin{remark}
\label{R2}
  In Theorem \ref{T2}, if we choose $h(t)=t^s$, $s\in (0,1],$  then  \eqref{2.3}
reduces to \eqref{1.7}.
\end{remark}

 A different approach leads to the following result.

\begin{theorem}
\label{T3} Suppose that all the assumptions of Theorem \ref{T2}
are satisfied. Then%
\begin{align}
&\left| \int_{a}^{b}f(x)dx-\frac{b-a}{6}\left[ f(a)+4f\left( \frac{a+b}{%
2}\right) +f(b)\right] \right| \notag\\
\leq &\frac{\left( b-a\right) ^{4}}{6}\left( \frac{1}{192}\right) ^{1-\frac{%
1}{q}} \notag\\
&\times \left\{ \left[ \left( \int_{0}^{\frac{1}{2}%
}t^{2}\left( \frac{1}{2}-t\right) h(t) dt\right)\left| f'''(a)\right| ^{q}+\left( \int_{0}^{\frac{1}{2}%
}t^{2}\left( \frac{1}{2}-t\right) h(1-t) dt\right)\left| f'''(b)\right| ^{q}\right] ^{\frac{1}{q}%
}\right. \notag\\
&\left. +\left[ \left( \int_{0}^{\frac{1}{2}%
}t^{2}\left( \frac{1}{2}-t\right) h(1-t) dt\right)\left| f'''(a)\right| ^{q}+\left( \int_{0}^{\frac{1}{2}%
}t^{2}\left( \frac{1}{2}-t\right) h(t) dt\right)\left| f'''(b)\right| ^{q}\right] ^{\frac{1}{q}}\right\}.  \label{2.4}
\end{align}
\end{theorem}

\begin{proof}
From Lemma \ref{L1} and using the well-known power-mean inequality we
have%
\begin{align*}
&\left| \int_{a}^{b}f(x)dx-\frac{b-a}{6}\left[ f(a)+4f\left( \frac{a+b}{%
2}\right) +f(b)\right] \right| \\
\leq &\frac{\left( b-a\right) ^{4}}{6}\left\{ \left( \int_{0}^{\frac{1}{2}%
}t^{2}\left( \frac{1}{2}-t\right) dt\right) ^{1-\frac{1}{q}}\left( \int_{0}^{%
\frac{1}{2}}t^{2}\left( \frac{1}{2}-t\right) \left| f'''(ta+(1-t)b)\right| ^{q}dt\right) ^{\frac{1}{q}}\right. \\
&\left. +\left( \int_{\frac{1}{2}}^{1}(t-1)^{2}\left( t-\frac{1}{2}\right)
dt\right) ^{1-\frac{1}{q}}\left( \int_{\frac{1}{2}}^{1}(t-1)^{2}\left( t-%
\frac{1}{2}\right) \left| f'''(ta+(1-t)b)\right| ^{q}dt\right) ^{\frac{1}{q}}\right\} .
\end{align*}%
Since $\left| f'''\right| ^{q}$ is $s-$%
convex, we have
\begin{align*}
&\int_{0}^{\frac{1}{2}}t^{2}\left( \frac{1}{2}-t\right) \left|
f'''(ta+(1-t)b)\right| ^{q}dt \\
\leq &\int_{0}^{\frac{1}{2}}t^{2}\left( \frac{1}{2}-t\right) \left(
h(t)\left| f'''(a)\right|^{q}
+h(1-t)\left| f'''(b)\right|^{q} \right)
dt \\
=& \left( \int_{0}^{\frac{1}{2}%
}t^{2}\left( \frac{1}{2}-t\right) h(t) dt\right)\left| f'''(a)\right| ^{q}+\left( \int_{0}^{\frac{1}{2}%
}t^{2}\left( \frac{1}{2}-t\right) h(1-t) dt\right)\left| f'''(b)\right| ^{q}
\end{align*}%
and
\begin{align*}
&\int_{\frac{1}{2}}^{1}(t-1)^{2}\left( t-\frac{1}{2}\right) \left|
f'''(ta+(1-t)b)\right| ^{q}dt \\
\leq &\int_{\frac{1}{2}}^{1}(t-1)^{2}\left( t-\frac{1}{2}\right) \left(
h(t)\left| f'''(a)\right|^{q}
+h(1-t)\left| f'''(b)\right|^{q} \right)
dt \\
=&\left( \int_{0}^{\frac{1}{2}%
}t^{2}\left( \frac{1}{2}-t\right) h(1-t) dt\right)\left| f'''(a)\right| ^{q}+\left( \int_{0}^{\frac{1}{2}%
}t^{2}\left( \frac{1}{2}-t\right) h(t) dt\right)\left| f'''(b)\right| ^{q}.
\end{align*}%
Hence,    the  proof of   \eqref{2.4} is complete.
\end{proof}

\begin{remark}
\label{R3}
  In Theorem \ref{T3}, if we choose $h(t)=t^s$, $s\in (0,1],$  then  \eqref{2.4}
reduces to \eqref{1.8}.
\end{remark}

\section{Simpson type inequalities for $(\alpha,m)-$convex functions}

We use the following modified identity:

\begin{lemma} \cite[Lemma 2]{oak2011}
\label{lem 2.1} Let $f:I\rightarrow
\mathbb{R}
$ be a function such that $f'''$be absolutely
continuous on $I^{\circ }$, the interior of I. Assume that $a,b\in I^{\circ
},$ with $a<b$, $m\in (0, 1]$ and $f'''\in L_1[a,b].$ Then, the
following equality holds:%
\begin{align*}
 &\int_{a}^{mb}f(x)dx-\frac{mb-a}{6}\left[ f(a)+4f\left( \frac{a+mb}{2}\right)
+f(mb)\right]\\
= &\left( mb-a\right) ^{4}\int_{0}^{1}p(t)f'''(ta+m(1-t)b)dt,
\end{align*}%
where%
\begin{equation*}
p(t)=\left\{
\begin{array}{ll}
\frac{1}{6}t^{2}\left( t-\frac{1}{2}\right), &t\in \lbrack 0,%
\frac{1}{2}], \medskip
\\
\frac{1}{6}(t-1)^{2}\left( t-\frac{1}{2}\right), &t\in (\frac{1}{2%
}, 1].
\end{array}%
\right.
\end{equation*}
\end{lemma}

Using this lemma, we can obtain the following  inequalities for $(\alpha,m)-$convex functions.

\begin{theorem}
  \label{T4} Let $f:I\subset \lbrack 0,b^{\ast }]\rightarrow
\mathbb{R}
,$ be a differentiable function on $I^{\circ }$ such that $f'''\in L_1[a,b]$ where $a,b\in I$ with $a<b,$ $b^{\ast }>0.$ If $%
\left| f'''\right| ^{q}$ is $(\alpha,m)-$convex
on $[a,b]$ for $(\alpha ,m)\in \lbrack
0,1]^{2},$ $q>1$ with $\frac{1}{p}+\frac{1}{q}=1,$ then
\begin{align}
&\left| \int_{a}^{mb}f(x)dx-\frac{mb-a}{6}\left[ f(a)+4f\left( \frac{a+mb}{2}\right)
+f(mb)\right] \right| \notag\\
\leq &\frac{\left(m b-a\right) ^{4}}{96} \left( \frac{\Gamma (2p+1)\Gamma (p+1)%
}{ \Gamma (3p+2)}\right) ^{\frac{1}{p}}  \left\{ \left[\frac{\left| f'''(a)\right|
^{q}+m[2^\alpha(1+\alpha)-1]\left| f'''(b)\right| ^{q}}{2^{\alpha}(1+\alpha)} \right] ^{%
\frac{1}{q}}\right. \notag\\
&\left. +\left[ \frac{(2^{1+\alpha}-1)\left| f'''(a)\right|
^{q}+m[2^{\alpha}(1+\alpha)-(2^{1+\alpha}-1)]\left| f'''(b)\right| ^{q}}{2^{\alpha}(1+\alpha)} \right] ^{\frac{1}{q}%
}\right\}.   \label{3.1}
\end{align}
\end{theorem}

\begin{proof}
From Lemma \ref{lem 2.1} and using H\"{o}lder's inequality we have%
\begin{align*}
&\left| \int_{a}^{mb}f(x)dx-\frac{mb-a}{6}\left[ f(a)+4f\left( \frac{a+mb%
}{2}\right) +f(mb)\right] \right| \\
\leq &\frac{\left( mb-a\right) ^{4}}{6}\left\{ \left( \int_{0}^{\frac{1}{2}%
}\left( t^{2}\left( \frac{1}{2}-t\right) \right) ^{p}dt\right) ^{\frac{1}{p}%
}\left( \int_{0}^{\frac{1}{2}}\left| f'''(ta+m(1-t)b)\right| ^{q}dt\right) ^{\frac{1}{q}}\right. \\
&\left. +\left( \int_{\frac{1}{2}}^{1}\left( (t-1)^{2}\left( t-\frac{1}{2}%
\right) \right) ^{p}dt\right) ^{\frac{1}{p}}\left( \int_{\frac{1}{2}%
}^{1}\left| f'''(ta+m(1-t)b)\right|
^{q}dt\right) ^{\frac{1}{q}}\right\} .
\end{align*}%
Due to the $(\alpha,m)-$convexity of $\left| f'''\right| ^{q},$ we have%
\begin{align*}
\int_{0}^{\frac{1}{2}}\left| f'''(ta+m(1-t)b)\right| ^{q}dt \leq &\int_{0}^{\frac{1}{2}}\left[
t^\alpha\left| f'''(a)\right|
^{q}+m(1-t^\alpha)\left| f'''(b)\right| ^{q}%
\right] dt \\
=&\frac{\left| f'''(a)\right|
^{q}+m[2^\alpha(1+\alpha)-1]\left| f'''(b)\right| ^{q}}{2^{1+\alpha}(1+\alpha)}
\end{align*}%
and
\begin{align*}
\int_{\frac{1}{2}}^{1}\left| f'''(ta+m(1-t)b)\right| ^{q}dt \leq &\int_{\frac{1}{2}}^{1}\left[
t^\alpha\left| f'''(a)\right|
^{q}+m(1-t^\alpha)\left| f'''(b)\right| ^{q}%
\right] dt \\
=&\frac{(2^{1+\alpha}-1)\left| f'''(a)\right|
^{q}+m[2^{\alpha}(1+\alpha)-(2^{1+\alpha}-1)]\left| f'''(b)\right| ^{q}}{2^{\alpha}(1+\alpha)}.
\end{align*}%
The proof of \eqref{3.1} is complete by combining the above inequalities and \eqref{2.3'}.
\end{proof}

\begin{remark}
\label{rem 2.1} In Theorem \ref{T4}, if we choose $\alpha=1,$ we get the
inequality in  \cite[Theorem 4]{oak2011}.
\end{remark}

\begin{theorem}
\label{T5} Let the assumptions of Theorem \ref{T4} hold with $%
q\geq 1.$ Then
\begin{align}
&\left| \int_{a}^{mb}f(x)dx-\frac{mb-a}{6}\left[ f(a)+4f\left( \frac{a+mb%
}{2}\right) +f(mb)\right] \right| \notag\\
\leq &\frac{\left( mb-a\right) ^{4}}{1152}\left\{ \left( \frac{12\left|
f'''(a)\right| ^{q}+m[2^\alpha(3+\alpha)(4+\alpha)-12]\left| f'''(b)\right| ^{q}}{2^\alpha(3+\alpha)(4+\alpha)}\right) ^{\frac{1}{q}}\right.\notag\\
&\left.+\left(\frac{12[\alpha^2+11\alpha+34-2^{4+\alpha}(2-\alpha)]}{2^\alpha(1+\alpha)(2+\alpha)(3+\alpha)(4+\alpha)} \left| f'''(a)\right| ^{q} \right.\right. \notag\\
&\left.\left.+ m\left[1-\frac{12[\alpha^2+11\alpha+34-2^{4+\alpha}(2-\alpha)]}{2^\alpha(1+\alpha)(2+\alpha)(3+\alpha)(4+\alpha)}\right]\left| f'''(b)\right| ^{q}  \right)
^{\frac{1}{q}}\right\}.
\end{align}
\end{theorem}

\begin{proof}
From Lemma \ref{lem 2.1}, using the well known power-mean inequality and $(\alpha,m)- $convexity of $\left| f'''\right|
^{q}$, we have
\begin{align*}
&\left| \int_{a}^{mb}f(x)dx-\frac{mb-a}{6}\left[ f(a)+4f\left( \frac{a+mb%
}{2}\right) +f(mb)\right] \right| \\
\leq &\frac{\left( b-a\right) ^{4}}{6}\left\{ \left( \int_{0}^{\frac{1}{2}%
}t^{2}\left( \frac{1}{2}-t\right) dt\right) ^{1-\frac{1}{q}}\left( \int_{0}^{%
\frac{1}{2}}t^{2}\left( \frac{1}{2}-t\right) \left| f'''(ta+m(1-t)b)\right| ^{q}dt\right) ^{\frac{1}{q}}\right.
\\
&\left. +\left( \int_{\frac{1}{2}}^{1}(t-1)^{2}\left( t-\frac{1}{2}\right)
dt\right) ^{1-\frac{1}{q}}\left( \int_{\frac{1}{2}}^{1}(t-1)^{2}\left( t-%
\frac{1}{2}\right) \left| f'''(ta+m(1-t)b)\right| ^{q}dt\right) ^{\frac{1}{q}}\right\} \\
\leq &\frac{\left( b-a\right) ^{4}}{6}\left\{ \left( \int_{0}^{\frac{1}{2}%
}t^{2}\left( \frac{1}{2}-t\right) dt\right) ^{1-\frac{1}{q}}\left( \int_{0}^{%
\frac{1}{2}}t^{2}\left( \frac{1}{2}-t\right) \left[ t^\alpha\left| f'''(a)\right| ^{q}+m(1-t^\alpha)\left| f'''(b)\right| ^{q}\right] dt\right) ^{\frac{1}{q}}\right. \\
&\left. +\left( \int_{\frac{1}{2}}^{1}(t-1)^{2}\left( t-\frac{1}{2}\right)
dt\right) ^{1-\frac{1}{q}}\left( \int_{\frac{1}{2}}^{1}(t-1)^{2}\left( t-%
\frac{1}{2}\right) \left[ t^\alpha\left| f'''(a)\right| ^{q}+m(1-t^\alpha)\left| f'''(b)\right| ^{q}\right] dt\right) ^{\frac{1}{q}}\right\} .
\end{align*}%
By using the fact that
$$\int_{0}^{\frac{1}{2}}t^{2}\left( \frac{1}{2}-t\right)t^\alpha dt=\frac{1}{16\times 2^\alpha(3+\alpha)(4+\alpha)} ,$$
$$\int_{0}^{\frac{1}{2}}t^{2}\left( \frac{1}{2}-t\right)(1-t^\alpha) dt=\frac{2^\alpha(3+\alpha)(4+\alpha)-12}{192\times 2^\alpha(3+\alpha)(4+\alpha)} ,$$
$$\int_{\frac{1}{2}}^{1}(t-1)^{2}\left( t-\frac{1}{2}\right)t^\alpha dt=\frac{\alpha^2+11\alpha+34-2^{4+\alpha}(2-\alpha)}{16\times 2^\alpha(1+\alpha)(2+\alpha)(3+\alpha)(4+\alpha)}$$
and
$$\int_{\frac{1}{2}}^{1}(t-1)^{2}\left( t-\frac{1}{2}\right)(1-t^\alpha) dt=\frac{2^\alpha(1+\alpha)(2+\alpha)(3+\alpha)(4+\alpha)-12[\alpha^2+11\alpha+34-2^{4+\alpha}(2-\alpha)]}{192\times 2^\alpha(1+\alpha)(2+\alpha)(3+\alpha)(4+\alpha)} ,$$
we obtain%
\begin{align*}
&\left| \int_{a}^{mb}f(x)dx-\frac{mb-a}{6}\left[ f(a)+4f\left( \frac{a+mb%
}{2}\right) +f(mb)\right] \right| \\
\leq &\frac{\left( mb-a\right) ^{4}}{6}\left( \frac{1}{192}\right) ^{1-\frac{%
1}{q}}\left\{\left( \frac{12\left|
f'''(a)\right| ^{q}+m[2^\alpha(3+\alpha)(4+\alpha)-12]\left| f'''(b)\right| ^{q}}{192\times 2^\alpha(3+\alpha)(4+\alpha)}\right) ^{\frac{1}{q}%
}
\right. \\
&\left.+\left(\frac{\alpha^2+11\alpha+34-2^{4+\alpha}(2-\alpha)}{16\times 2^\alpha(1+\alpha)(2+\alpha)(3+\alpha)(4+\alpha)} \left| f'''(a)\right| ^{q} \right.\right. \\
&\left.\left.+ m\left[\frac{1}{192}-\frac{\alpha^2+11\alpha+34-2^{4+\alpha}(2-\alpha)}{16\times 2^\alpha(1+\alpha)(2+\alpha)(3+\alpha)(4+\alpha)}\right]\left| f'''(b)\right| ^{q}  \right)
^{\frac{1}{q}}\right\},
\end{align*}%
which implies the desired result.
\end{proof}

\begin{remark}
\label{rem 2.2} In Theorem \ref{T5}, if we choose $\alpha=1,$ we have the
inequality in \cite[Theorem 5]{oak2011}.
\end{remark}

\end{document}